\documentclass[11pt]{article}
\usepackage{amsmath}
\usepackage{amsthm}
\usepackage{amssymb}

\newtheorem{theo}{Theorem}[section]

\newtheorem{prop}[theo]{Proposition}
\newtheorem{lemma}[theo]{Lemma}

\newtheorem{claim}[theo]{Claim}

\begin{document}
\date{}

\title{
Tur\'an graphs with bounded matching number 
}

\author{Noga Alon
\thanks{Princeton University,
Princeton, NJ, USA 
and
Tel Aviv University, Tel Aviv,
Israel.
Email: {\tt nalon@math.princeton.edu}.
Research supported in part by
NSF grant DMS-2154082 and by USA-Israel BSF grant 2018267.
}
\and
Peter Frankl
\thanks{
R\'enyi Institute, Budapest, Hungary.
Email: {\tt peter.frankl@gmail.com}. 
}
}

\maketitle
\begin{abstract}
We determine the maximum possible number of edges of a graph with
$n$ vertices, matching number at most $s$ and clique number at most $k$
for all admissible values of the parameters.
\end{abstract}

\section{The main result}
The clique number of a graph $G$ is the maximum number of vertices in 
a complete subgraph of it. The matching number of $G$ is the maximum
cardinality of a matching in $G$.
Two classical results in Extremal Graph Theory are Tur\'an's Theorem
\cite{Tu}
determining the maximum number of edges $t(n,k)$ of a graph on $n$ vertices
with clique number at most $k$,
and the Erd\H{o}s-Gallai Theorem \cite{EG}, determining the
maximum possible number of edges of a graph with $n$ vertices and
matching number at most $s$. 

In this note we prove a common generalization. Call a graph complete
$k$-partite if its vertex set consists of $k$ pairwise disjoint  sets
and two vertices are adjacent iff they belong to distinct classes. 
Note that we allow some vertex classes to be empty.
Let $T(n,k)$ denote the
complete $k$-partite graph  with $n$ vertices in which the 
sizes of the vertex classes
are as equal as possible, and let
$t(n,k)$ denote its number of edges. Let $G(n,k,s)$ denote the complete
$k$-partite graph on $n$ vertices consisting of $k-1$ vertex classes 
of sizes as equal as possible whose total size is $s$, and one additional
vertex class of size $n-s$. Let $g(n,k,s)$ denote the number of its
edges. 

Our main result is the following.
\begin{theo}
\label{t11}
For all $n \geq 2s+1$ and every $k$, the maximum
possible number of edges of a graph on $n$ vertices with clique
number at most $k$ and matching number at most $s$ is the maximum between
the Tur\'an number $t(2s+1,k)$ and the  number $g(n,k,s)$ defined
above. (For $n \leq 2s+1$ the maximum is clearly $t(n,k)$).
\end{theo}

\section{Proof}
Let $G=(V,E)$ be a graph on $n \geq 2s+1$ vertices with matching number 
at most $s$
and clique number at most $k$ having the maximum possible number 
of edges.
By the Tutte-Berge Theorem or the Edmonds-Gallai Theorem, 
cf., e.g. \cite{LP},
there is a set of vertices $B$, $|B|=b$ so that each of the connected
components $A_1, A_2, \ldots ,A_m$ of $G-B$ is odd, 
and so that if the sizes of these
components are 
$$
|A_1|=a_1 \geq |A_2| =a_2 \geq \cdots \geq |A_m|=a_m \geq 1
$$ 
then
$$
b+\sum_{i=1}^m (a_i-1)/2=s
$$
and 
$$
b+\sum_{i=1}^m a_i=n.
$$
Among all such graphs with the maximum possible number of edges
assume that $G$ is one for which the sum $\sum_{i=1}^m a_i^2$
is maximum.

We use the following standard notation.
For any vertex $v$ of $G$, $N(v)$ denotes its set of neighbors.
If $C$ is a set of vertices of $G$, put $N_C(v)=N(v) \cap C$.
$G_C$ denotes the induced subgraph of $G$ on $C$.

We first prove the following lemma, which is a simple
consequence of the
Zykov symmetrization method introduced in 
\cite{Zy}. For completeness we include a short proof.
\begin{lemma}
\label{l20}
Every two non-adjacent vertices of $B$ have the same neighborhood.
\end{lemma}
\begin{proof}
It suffices to show that non-adjacency is an equivalence relation 
on $B$. Indeed, this relation is trivially reflexive and 
symmetric. Suppose it is not transitive,
then there are three distinct vertices
$u,v,w$  in $B$ so that $uv, uw$ are non-edges but $vw$ is an edge.
If the degree $d(u)$ of $u$  is smaller than $d(v)$, then replacing
the neighborhood of $u$ by that of $v$ the number of edges increases.
The clique number does not increase, as any new clique $K$ must contain
$u$, but then it cannot contain $v$, and $(K-\{u\}) \cup \{v\}$ is 
a clique of the same size before the replacement. The matching number
also stays at most $s$, as  demonstrated by the set of vertices $B$
after the replacement. Thus, by the assumption that $G$ has 
a maximum possible number of edges it follows that $d(u) \geq d(v)$.
The same argument shows that $d(u) \geq d(w)$. But in this case the graph
obtained by replacing the neighborhood of $v$ by that of  $u$ and the 
neighborhood of $w$ by that of $u$ provides the desired contradiction. 
Indeed, it has more edges than $G$, clique number at most that of $G$,
and matching number at most $s$. This completes the proof of the lemma.
\end{proof}

\begin{lemma}
\label{l21}
$a_i=1$ for all $2 \leq i \leq m$.
\end{lemma}

\begin{proof}
By  Lemma \ref{l20} every two
non-adjacent vertices of $B$ have the same neighborhood.
Since $G$ contains no clique of size $k+1$ this means that 
$G_B$ is a complete $k$-partite graph. Let
$B_1,B_2, \ldots ,B_k$ be the vertex classes of this induced
subgraph, with $|B_1| \geq |B_2| \geq \ldots \geq |B_k|$
(where some of these classes may be empty).
\begin{claim}
\label{c22}
Without loss of generality we may assume that for every
$1 \leq i \leq m$ there is a vertex $v_i \in A_i$ 
which has no neighbor in $B_k$.
\end{claim}
\vspace{0.2cm}

\noindent
{\bf Proof of Claim:}\, If $B_k=\emptyset$ this is surely true.
We can thus assume that $|B_1| \geq |B_2| \geq \ldots \geq |B_k| 
\geq 1$. Since the size $w(G)$ of the largest clique of $G$ is at most
$k$, no vertex in $A_i$ is adjacent to a member of each $B_j$,
$1 \leq j \leq k$. 
If all vertices of $A_i$ are adjacent to 
$B_k$ (to all of it, as all vertices in $B_k$ have the same
neighborhood) 
we can swap $B_j$ and $B_k$ in the neighborhood 
of each $v \in A_i$ leaving it connected to both if it has been
connected to both, and leaving it connected only to $B_j$
if it has been connected only to $B_k$. This can only increase
the number of edges, as $|B_k| \leq |B_j|$. Choosing $j$ so that
some vertex $v \in A_i$ has no neighbors in $B_j$ gives the desired
assertion of the claim. Note also that swapping $B_j$ and $B_k$
as above cannot increase the size of the maximum clique as any new
clique created this way includes a vertex of $B_j$, some vertices
of $A_i$,
and no vertex of $B_k$. Replacing the vertex from $B_j$ by any one
of $B_k$ gives a clique of the same size in the graph before the swap. 
Since the matching number also stays at most $s$, as shown by
$B$, this
completes the proof of the claim. \hfill  $\Box$

Returning to the proof of the lemma assume it is false and
$a_1 \geq a_2 \geq 3$. Let $v_1 \in A_1$ and $v_2 \in A_2$ 
be as in the claim.
Now modify $G$ into $G'$ by defining
$A_1'=A_1 \cup A_2 \setminus \{v_2\}, A_2'=\{v_2\}$, keeping $B'=B$
and only changing the edges incident with $v_1$ and $v_2$ as follows.
The new neighborhood of $v_1$ is 
$$
N'(v_1)=N_{A_1}(v_1) \cup N_{A_2}(v_2) \cup (N_B(v_1) \cap N_B( v_2)).
$$
The new neighborhood of $v_2$ is $N_B(v_1) \cup N_B(v_2)$.

The total number of edges is unchanged, and $(a_1,a_2)$ changed to
$(a_1+a_2-1,1)$ implying that the matching number stays at most
$s$, as both $a_1+a_2-1$ and $1$ are odd. 
The clique number stays at most $k$. Indeed, any new
clique containing $v_2$ is of size at most $k$ since
neither $v _1$ nor $v_2$ are adjacent to $B_k$ in $G$. Any new clique 
$K$ in $G'$ 
containing $v_1$ contains in $A_1'$ either only vertices of $A_1$
or only vertices of $A_2-\{v_2\}$ (in addition to $v_1$). In the first
case, since $N'_B(v_1) \subset N_B(v_1)$, the same clique appears also
in $G$. In the second case, since $N'_B(v_1)\subset N_B(v_2)$,
$(K-\{v_1\} )\cup\{v_2\}$ is a clique in $G$, of the same size as
$K$.  Since
$(a_1+a_2-1)^2+1^2> a_1^2+a_2^2$ this yields a contradiction and
completes the proof of the lemma.
\end{proof} 

By the lemma it follows that $a_1=2s-2b+1$. We consider several
possible cases, as follows.
\vspace{0.2cm}

\noindent
{\bf Case 1:}\, $b=0$. In this case $a_1=2s+1$ and all other
vertices of $G$ are isolated, showing that the number of edges
is at most $t(2s+1,k)$.
\vspace{0.2cm}

\noindent
{\bf Case 2:}\, $b=s$. In this case $a_1=1$ and all the components
of $G-B$ are isolated vertices. The induced subgraph of $G$ on
the union of $B$ with arbitrarily chosen additional 
$\lfloor s/(k-1) \rfloor$ components (each of size $1$) has
at most $t(s+\lfloor s/(k-1) \rfloor,k)$ edges. Any other 
vertex can be connected only to the vertices of $B$, namely
has degree at most $s$, and this gives exactly the number
$g(n,k,s)$ for the total number of edges.
\vspace{0.2cm}

\noindent
{\bf Case 3:}\, $|B|+a_1=2s-b+1 \leq s+\lfloor s/(k-1) 
\rfloor$.  This is similar to Case 2. The induced  subgraph of
$G$ on the union of $B$ with $A_1$ and with additional components
having total size $s+\lfloor s/(k-1) \rfloor$  spans at most
$t(s+\lfloor s/(k-1) \rfloor,k)$ edges. Any other vertex has
degree at most $b \leq s$ and the desired estimate follows
as before.
\vspace{0.2cm}

\noindent
{\bf Case 4:}\, $|B|+a_1=2s-b+1 > s+\lfloor s/(k-1) \rfloor $.
In this case $0 \leq b \leq s-\lfloor s/(k-1 \rfloor$.
Define
$$
f(b) =t(2s-b+1,k)+b (n-2s+b-1).
$$
The number of edges of $G$ is clearly at most $f(b)$.
Indeed, the induced subgraph on $B \cup A_1$
spans at most $t(2s-b+1,k)$ edges, and all remaining vertices
have degrees at most $b$.
We claim that in the relevant range of $b$, $f(b+1)-f(b)$ is an increasing
function of $b$. Note that the claim here is not that the function
$f(b)$ itself is increasing (in general it is not), but that its (discrete) 
derivative is increasing, that is, it is a discrete convex function. 
To prove the claim note that
$$
f(b+1)-f(b)=n-2s+2b-[t(2s-b+1,k)-t(2s-b,k)]
$$
When $b$ increases by $1$, the term $(n-2s+2b)$ increases 
by $2$, and the term 
$$
t(2s-b+1,k)-t(2s-b,k)
$$ 
can only decrease (as it is the
difference in the total size of the largest $k-1$ classes among the $k$
nearly equal classes of the corresponding 
Tur\'an graphs, and this quantity can only decrease
(by at most $1$) when decreasing the number of vertices $2s-b$ by $1$).
This shows that $f(b+1)-f(b)$ is increasing in the range above. 
Therefore,
if $f(b)$ obtains a maximum at some $b>0$ in this range, that is,
$f(b) \geq f(b-1)$ then it must be that the maximum is obtained
at the largest possible $b$ in this range, which is
$b=s-\lfloor s/(k-1 \rfloor$.  But this is covered by Case 3,
completing the proof.  $\Box$

\section{Extension}
It may be interesting to extend Theorem \ref{t11} by replacing the
forbidden clique $K_{k+1}$ 
by other forbidden subgraphs. This means to determine 
the maximum possible number of edges of an $H$-free graph on $n$ vertices
with matching number at most $s$. Recall that a graph $H$ is 
color-critical if it contains an edge whose deletion decreases
its chromatic number.  It is not difficult 
to prove the following, combining the initial part of
our proof here
with the known result of Simonovits \cite{Si} about the
Tur\'an numbers of color-critical graphs. Here we include 
a slightly simpler proof which avoids the application of the
Tutte-Berge or the Gallai-Edmonds Theorems.
\begin{prop}
\label{p31}
For every fixed color-critical graph $H$ of chromatic number 
$k+1>2$, any $s>s_0(H)$ and any $n>n_0(s)$, the maximum possible number 
of edges of an $H$-free graph on $n$ vertices with matching number
at most $s$ is $g(n,k,s)$.
\end{prop}
\begin{proof}
The graph $G(n,k,s)$ described before the  statement of the 
main theorem is $k$ chromatic and hence $H$-free. Since its
matching number is $s$ this implies that the number of edges of
this graph, which is $g(n,k,s)$,  is a lower bound for the
maximum considered in the proposition.
To prove the upper bound, let
$H,k,s$ be as above and let 
$G$ be an $H$-free graph on $n$ vertices with matching number at most
$s$ having the maximum possible number
of edges. Assume, further, that $s$ is sufficiently
large as a function of $H$ and that $n$ is sufficiently
large as a function of $s$. 

Note, first, that $G$ cannot contain more than $s$ vertices of degrees
exceeding $2s$. Indeed, otherwise let $\{x_1,x_2, \ldots x_{s+1}\}$
be $s+1$ such vertices. For each $x_i$, in order,
let $y_i$ be an arbitrarily chosen neighbour of $x_i$ which differs
from all $x_j$ and all previously chosen $y_j$. As there are only
$s+i-1 \leq 2s $ such forbidden vertices  (we do not have to count the
vertex $x_i$ itself) there is always a choice for $y_i$. This gives
a matching of size $s+1$, contradicting the assumption.

Let $X$ be the set of all vertices of degree exceeding $2s$. By the
paragraph above $|X| \leq s$. Put $Y=V-X$. In the induced subgraph of
$G$ on $Y$ every degree is at most $2s$ and there is no matching of
size $s+1$, hence by Vizing's Theorem the number of vertices in this
induced subgraph is at most $(2s+1)s$. As the total number of edges
incident with the vertices in $X$ is smaller than $|X|n$
(with room to spare) it follows that
if $|X|<s$ then the number of edges of $G$ is smaller than
$(s-1)n +2(s+1)s$. This is smaller than $g(n,k,s)$
for $n$ exceeding, say, $3s^2$ (we make no 
attempt to optimize $n_0(s)$), 
showing that we may assume that $|X|=s$.

We claim that $Y=V-X$ is an independent  set in $G$. Indeed, if it contains
an edge $z_1z_2$ we can, as before, use the fact that the degree of
each vertex of $X$ exceeds $2s$ to pick distinct $y_i \in
Y-\{z_1,z_2\}$ so that $x_iy_i$ is an edge for each $i$, contradicting
again the assumption about the matching number. Thus $Y$ is indeed
independent.

Let $Z$ be an arbitrary subset of $Y=V-X$
of size $m=\lfloor s/(k-1) \rfloor$.
By the result of Simonovits, for $s>s_0(H)$
the induced subgraph of $G$ on $X \cup Z$ contains at most 
$t(s+m,k)$ edges. In
addition, all other edges of $G$ are incident with the vertices of $X$,
as $Y$ is independent.  Therefore, the total number of edges of $G$
is at most the number of edges of the graph obtained from
the Tur\'an graph
$T_{s+m,k}$ on a set $X \cup Z$ of $s+m$ vertices
in which $Z$ is (one of) the  smallest vertex classes
by adding to it an independent set of size $n-s-m$
and by connecting each of its vertices to the 
$s$ vertices of $X$. It is easy
to see that this graph is isomorphic to the graph $G(n,k,s)$,
completing the proof. 
\end{proof}

\end{document}